\documentclass[reqno]{amsart}
\usepackage{graphicx, enumerate,amssymb}

\newtheorem{thm}{Theorem}[section]

\newtheorem{lem}[thm]{Lemma}
\newtheorem{prop}[thm]{Proposition}

\theoremstyle{definition}
\newtheorem{defn}[thm]{Definition}
\theoremstyle{remark}

\numberwithin{equation}{section}
\newcommand{\norm}[1]{\left\Vert#1\right\Vert}

\newcommand{\abs}[1]{\left\vert#1\right\vert}

\newcommand{\cx}{{\mathbb{C}}}
\newcommand{\dist}{{\mathrm { dist}}}
\newcommand{\tensor}{\otimes}

\newcommand{\sng}[1]{{#1}^{\rm sng}}
\newcommand{\reg}[1]{{#1}^{\rm reg}}

\newcommand{\psh}{plurisubharmonic }

\newcommand{\dbar}{\overline{\partial}}

\newcommand{\tmop}[1]{\ensuremath{\operatorname{#1}}}
\renewcommand{\Re}{\tmop{Re}}

\title{Condition R and holomorphic mappings of domains with generic corners}
\thanks{Kaushal Verma  was supported in part by the DST SwarnaJayanti fellowship 2009--2010 and the UGC--CAS Grant}
\author{Debraj Chakrabarti}
\address{Department of Mathematics, Central Michigan University, Mt. Pleasant,  MI 48859,  USA}
\email{chakr2d@cmich.edu}
\author{Kaushal Verma}
\address{Department of Mathematics, Indian Institute of Science, Bengaluru-560012, India}
\email{kverma@math.iisc.ernet.in}
\begin{document}
\begin{abstract}A piecewise smooth domain is said to have generic corners if the corners are generic CR manifolds.
It is shown that a biholomorphic mapping from a piecewise smooth pseudoconvex domain with generic corners in complex Euclidean space
that satisfies Condition R to another domain  extends as a smooth diffeomorphism of the  respective 
closures if and only if the target domain is also piecewise smooth  with generic corners and satisfies Condition R.
Further it is shown that a proper map from a domain with generic corners satisfying Condition R to a product domain of the same dimension 
extends continuously to the closure of the source domain in such a way that the extension is smooth on the smooth part of the boundary.
In particular, the existence of such a proper mapping  forces the smooth part of the  boundary of the source to be 
Levi degenerate.
\end{abstract}

\maketitle
\section{Introduction}
\noindent The question of continuous or smooth  extension to the boundary of holomorphic maps is of central importance in complex analysis.
One significance of such extension lies in the fact that it reduces the  difficult problem of classification of domains in $\cx^n$, $n\geq 2$ up to biholomorphism,
or the more general problem of deciding the existence of a proper map between two given domains,
to the problem of study of CR invariants of the boundary hypersurfaces. After the fundamental result in this direction  of Fefferman (\cite{fef}) giving smooth extension 
up to the boundary of a biholomorphic map between strictly pseudoconvex domains, there were obtained  far reaching generalizations to proper maps between 
smoothly bounded pseudoconvex  domains (e.g., \cite{BL, B2,Be1,Be2,BC,DF}.) In these investigations, the hypothesis on the source domain $D$ of the proper map  is 
that it satisfies {\em Condition R}:   the {\em Bergman projection}, 
the orthogonal projection from the Hilbert space $L^2(D)$ of square integrable functions to the closed subspace $\mathcal{H}(D)$
of holomorphic square  integrable functions,  maps a function smooth up to the boundary to a holomorphic function smooth up to the boundary.

\medskip

In this note we consider a class of piecewise smooth domains to which the techniques of Bell-Catlin-Diederich-Fornaess-Ligocka et al. mentioned above extend
in a natural way. By definition, a {\em piecewise smooth domain} is an intersection of finitely many smoothly bounded domains in which all 
possible boundary intersections are  transverse.  The class of domains we will be considering  are the {\em domains with generic corners} 
defined below. Such domains have been considered by various authors (see \cite{Ba, forst,We}.) In \cite{We}, Webster 
considered holomorphic mappings defined on domains with {\em real analytic} generic corners, and a 
reflection principle for such  corners was developed. These ideas were subsequently developed by Forstneri\v{c} (see \cite{forst}.)
A crucial estimate of Bell for holomorphic functions on smooth domains was generalized by Barrett to this class of domains
(see \cite{Ba}, and Lemma~\ref{lem-bell} below.)
In a previous article (\cite{prop}) we considered the extension of  proper mappings of equidimensional  products of smoothly bounded domains.
These products are examples of domains with generic corners,  and here we generalize  some of the results 
of \cite{prop} to the wider class. We now formally define these domains:
\begin{defn}\label{defn-main}
Let $\Omega$ be a bounded domain in $\cx^n$ that may be written as an 
intersection $\bigcap_{j=1}^N \Omega_j$ of smooth domains such that
\begin{enumerate}[(i)]
\item all intersections of the boundaries $b\Omega_j$ are transverse.
\item for each subset $S\subset\{1,\dots,N\}$ the intersection $B_S=\bigcap_{j\in S} b\Omega_j$, if non-empty,
is a CR manifold of CR-dimension $n-\abs{S}$.
\end{enumerate}
We call such a domain a {\em domain with generic corners.}
\end{defn}
Our first result is the following:

\noindent 

\begin{thm}\label{thm-biholo}
Suppose that $D\subset \cx^n$ is a pseudoconvex domain with generic corners which satisfies Condition R. If $G\subset\cx^n$ 
is a domain and $f:D\to G$ is a biholomorphic map, then the following are equivalent:

(1) $f$ extends as a $\mathcal{C}^\infty$-smooth diffeomorphism from $\overline{D}$ to $\overline{G} $.

 (2) $G$ is a domain with generic corners and satisfies Condition R.
 
\end{thm}

\noindent Therefore  the property of a domain that it satisfies Condition R  and has generic corners is invariant under holomorphic maps smooth up to the boundary.
As a result,  the classification of domains in this class is reduced to the study of the boundaries.  In Section~\ref{sec-examples} we consider some 
examples of domains satisfying the hypotheses of  Theorem~\ref{thm-biholo}. These are also the  hypotheses on the source domain 
$D$ in Theorem~\ref{thm-proper} below.  In a future work, we will consider further examples of this class of domains.

\medskip

\medskip

For a domain $\Omega$ with generic corners, let $b \sng{\Omega}\subset b \Omega$ consist of all those points that lie on the 
intersection of two or more boundaries $b \Omega_j$ and set $b \reg{\Omega} = b \Omega \setminus b\sng{ \Omega}$.

\begin{thm}\label{thm-proper}
Let $D \subset \mathbb C^n$ be a pseudoconvex domain with generic corners  and $G = G_1 \times G_2 \times \ldots \times G_k \subset \mathbb C^n$ a product domain where each $G_j 
\subset \mathbb C^{\mu_j}$ is smoothly bounded and $\mu_1 + \mu_2 + \ldots + \mu_k = n$. Assume that $D$ satisfies Condition R and let $f : D \rightarrow G$ be a proper holomorphic 
mapping. Then $f$ admits a continuous extension to $\overline{D}$ in such a way that the extension is $C^{\infty}$ smooth on $b \reg{D}$.
\end{thm}

\noindent It is possible to prove continuous extension of holomorphic maps between piecewise smooth domains under  hypotheses different from those used in Theorems~\ref{thm-biholo} and \ref{thm-proper}.
Piecewise smooth pseudoconvex domains that admit plurisubharmonic peak points on their boundaries were considered by Berteloot (\cite{Be}) and H\"{o}lder 
continuity at the boundary for proper holomorphic mappings between such domains was established. A similar result that relied on estimates for the Carath\'{e}odory metric on 
strictly pseudoconvex piecewise smooth domains was proved by Range (\cite{Ra}). 

One interesting question that Theorem~\ref{thm-proper} leaves unresolved is whether we can conclude from the hypotheses if
the source $D$ itself has a product structure, i.e., if there is a biholomorphic map $F:D\to F(D)$ onto a product domain $F(D)\subset\cx^n$,
where $F$ extends to a diffeomorphism from $\overline{D}$ to $\overline{F(D)}$. It would be interesting to know if this indeed is the case.

\medskip

\noindent {\em Acknowledgements:} We thank David Barrett for drawing attention to the class of domains with generic corners and
 pointing out that the results of \cite{prop} could be generalized to this class, and  Steven Bell for his helpful comments. Indeed
 much of this article is based on their ideas.  We thank the referee for  many helpful suggestions. Debraj Chakrabarti also thanks  Mythily Ramaswamy, 
 the Dean of TIFR CAM, Bangalore, for her support and encouragement of this research.

\section{Bell Operator} 

\noindent Let $\Omega$ be a domain with generic  
corners in $\cx^n$ and let  $N$ and $\Omega_j$ have the same meaning as
in Definition~\ref{defn-main}.  Suppose that $r_j$ (where $j=1,\dots, N$)
is a defining function of the domain $\Omega_j$, i.e., $r_j$ is a smooth function on $\cx^n$ 
such that $ \Omega_j= \{r_j<0\}$ and $dr_j$ is nonzero at each point of $b\Omega_j$. Then the conditions
(i) and (ii) in Definition~\ref{defn-main} may be rephrased as follows: for each point $p$ such that 
\[  r_{j_1}(p)= r_{j_2}(p)=\dots= r_{j_k}(p)=0\]
we have
\[ dr_{j_1}(p)\wedge dr_{j_2}(p)\wedge \dots \wedge dr_{j_k}(p)\not=0\]
and also
\begin{equation}\label{eq-c-trv}  
\dbar{}r_{j_1}(p)\wedge \dbar r_{j_2}(p)\wedge \dots \wedge \dbar r_{j_k}(p)\not=0.
\end{equation}

\begin{lem}[cf. Barrett \cite{Ba1,Ba}]\label{lem-bell} Let  $s=(s_1,\dots, s_N)$ be a tuple of non-negative integers. There is a linear differential operator
$\Phi^s$ with smooth coefficients defined on $\overline{\Omega}$ such that  for all $f\in \mathcal{C}^\infty(\overline{\Omega})$,
\begin{enumerate}[(i)]
\item $P\Phi^s f=P f$  and
\item $\abs{\Phi^s f(z)}\leq C \norm{f}_{\mathcal{C}^{\abs{s}}} d(z)^s$, where  $\abs{s}= \sum_{j=1}^N s_j$ and 
 \[ d(z)^s =d_1(z)^{s_1}\dots d_N(z)^{s_N},\]
where $d_j(z)$ is the distance from the point $z$ to $b\Omega_j$.
\end{enumerate}
\end{lem}
\begin{proof}Thanks to \eqref{eq-c-trv}, near  each $p\in \cx^n$ we can find $N$ vector fields $T_1^{(p)},\dots, T_N^{(p)}$  of type $(0,1) $ 
such that 
$T^{(p)}_j r_k \equiv\delta_{jk}$ in a neighborhood of $p$ when $r_j(p)=r_k(p)=0$. By a partition of unity argument, we obtain vector fields $T_j$ , $j=1,\dots N$
on $\cx^n$ of type $(0,1)$ such that $T_j r_k\equiv \delta_{jk}$ on a neighborhood $U_{jk}$ of $b\Omega_j\cap b\Omega_k$. (Note that if $j=k$,
this means that $T_j r_j\equiv 1$ near $b\Omega_j$.)

\medskip

For a  subset $S\subset \{1,\dots, N\}$,  let
\[ U_S = \bigcap_{j,k\in S} U_{jk}\setminus \bigcup_{\ell \not\in S} b\Omega_\ell.\]
Then the family $\{U_S\}$, as $S$ runs over all possible subsets of $\{1, 2, \ldots, N\}$ including the empty set, is an open cover of $\cx^n$. 
Let $\{\chi_S\}$ be a partition of unity subordinate to this cover. Let
\[ 
\langle f,g \rangle =\int_\Omega f\overline{g} dV
\]
denote the standard inner product on $L^2(\Omega)$, where $dV$ denotes Lebesgue measure on $\cx^n$.  Let $T_j^*$ denote the 
formal adjoint of the operator $T_j$ with respect to this inner product structure. Integration by parts shows that
$T_j^*=-\overline{(T_j+ {\rm div} \,T_j)}$, and is therefore also a first order operator with smooth coefficients. 
For $f\in \mathcal{C}^\infty(\overline{\Omega})$, we define the operator $\Phi^s$ by
\[ 
\Phi^s f = \sum_{S\subset \{1,\dots, N\}} \left( \prod_{j\in S} \frac{r_j^{s_j}}{s_j!}\right) \left(\prod_{j\in S}(T_j^*)^{s_j}
\right)\left(\chi_S f\right).
\]
Then $\Phi^s$ is a linear differential operator of order $\abs{s}$.  

\medskip

Note that for each $S\subset\{1,\dots, N\}$, the smooth function $\chi_S$ vanishes to infinite order along the set 
$\bigcup_{\ell \not\in S} b\Omega_\ell$, and therefore, for any multi-index $\alpha=(\alpha_1,\dots,\alpha_{2n})$, if
$D^\alpha$ is the partial derivative operator
\[ 
D^\alpha= \prod_{j=1}^n\left( \frac{\partial}{\partial x_j}\right)^{\alpha_{2j-1}}\left( \frac{\partial}{\partial y_j}\right)^{\alpha_{2j}},
\]
and $\sigma_j$ is a nonnegative integer for  each $j\not \in S$, we have an elementary estimate
\[ 
\abs{D^\alpha \chi_S(z)}\leq C_{\sigma,\alpha} \prod_{j\not \in S} \left(d_j(z)\right)^{\sigma_j},
\]
where $C_{\sigma,\alpha}$ is a constant independent of $z$. Therefore we have 
\[
\abs{\prod_{j\in S}(T_j^*)^{s_j}
\left(\chi_S f(z)\right)}\leq C_S \norm{f}_{\mathcal{C}^{\abs{s}}}\prod_{\ell \not \in S}\left(d_\ell(z)\right)^{s_\ell} .
\]
Since $r_j$ is comparable to $d_j$ for each $j$, we have that
\[
\abs{\Phi^s f (z)} 
\leq C\norm{f}_{\mathcal{C}^{\abs{s}}} \sum_{S\subset \{1,\dots, N\}} \left(\prod_{j \in S}\left(d_j(z)\right)^{s_j}\cdot \prod_{\ell \not \in S}\left(d_\ell(z)\right)^{s_\ell} \right),
\]
which proves part (ii) of the Lemma. 

\medskip

To prove part (i), it suffices to show that for $h\in L^2(\Omega)$  and $f\in \mathcal{C}^\infty(\overline{\Omega})$ we have
\[ \langle h, P \Phi^s f\rangle = \langle h, Pf\rangle.\]
Since $P$ is the (self-adjoint) orthogonal projection from $L^2(\Omega)$ onto  $\mathcal{H}(\Omega)= \mathcal{O}(\Omega)\cap L^2(\Omega)$,
this is equivalent to
\[ \langle  g,\Phi^s f \rangle = \langle g, f\rangle,\]
where  $g\in \mathcal{H}(\Omega)$. Now we have, 
\begin{align}
\langle g, \Phi^s f\rangle&=
\sum_{S\subset\{1,\dots,N\}} \left\langle  g, \left( \prod_{j\in S} \frac{r_j^{s_j}}{s_j!}\right) \left(\prod_{j\in S}(T_j^*)^{s_j}
 \right)\left(\chi_S f\right)\right\rangle\nonumber\\
&= \sum_{S\subset\{1,\dots,N\}} \left\langle \prod_{j\in S} \frac{r_j^{s_j}}{s_j!} g, \left(\prod_{j\in S}(T_j^*)^{s_j}
 \right)\left(\chi_S f\right)\right\rangle\label{eq-sum}\\
\end{align}
For $\epsilon>0$, let $\langle,\rangle_{\Omega_\epsilon}$ denote the standard $L^2$-inner product on the domain
\[ \Omega_\epsilon= \{ z\in \cx^n\colon r_j(z)<-\epsilon, 1\le j \le N\}.\]
Fix $S\subset \{1,\dots, N\}$, and first suppose that 
$S\not=\emptyset$. Denote by $j_0$ the smallest element of $S$.  Then we have, integrating by parts:
\begin{align}
 \left\langle \prod_{j\in S} \frac{r_j^{s_j}}{s_j!} g, \left(\prod_{j\in S}(T_j^*)^{s_j}
 \right)\left(\chi_S f\right)\right\rangle_{\Omega_\epsilon}
&= \left\langle T_{j_0}\left( \prod_{j\in S} \frac{r_j^{s_j}}{s_j!} g\right), \left((T_{j_0}^*)^{s_0-1}\cdot\prod_{j\in S\setminus\{j_0\}}(T_j^*)^{s_j}
 \right)\left(\chi_S f\right)\right\rangle_{\Omega_\epsilon}\nonumber\\
&- \sum_{k=1}^N \int_{b\Omega_k^\epsilon} \left( \prod_{j\in S} \frac{r_j^{s_j}}{s_j!} g\right)\cdot 
\overline{ \left((T_{j_0}^*)^{s_0-1}\cdot\prod_{j\in S\setminus\{j_0\}}(T_j^*)^{s_j} \right)\left(\chi_S f\right)}\cdot
\frac{T_{j_0} r_k}{\abs{dr_k}} dS,  \label{eq-omegae}
\end{align}
where $b\Omega_k^\epsilon = \{r_k= -\epsilon\}$.
In the boundary term, only the summand corresponding to $k=j_0$ is non-zero, since by construction,
$T_{j_0}r_k=\delta_{kj_0}$ in $U_S$. Using the Cauchy-Schwarz inequality, and the fact that $f\in \mathcal{C}^\infty(\overline{\Omega})$,
the square of the absolute-value of the boundary term may be estimated to be less than or equal to the quantity
\begin{align*} 
& \phantom{\leq}C \int_{b\Omega^\epsilon_{j_0} } \left( \prod_{j\in S} \frac{r_j^{s_j}}{s_j!}\right)^2 \abs{g}^2 dS \\
&\leq C\rq{} \epsilon^2 \int_{b\Omega^\epsilon_{j_0}}\abs{g}^2 dS,
\end{align*}
with $C$ and $C\rq{}$  independent of $\epsilon$. Since $g\in L^2(\Omega)$, we can find a sequence $\epsilon_i\to 0$ such that
$\int_{b\Omega^{\epsilon_i}_{j_0}}\abs{g}^2 dS=o(\epsilon_i^{-1})$. Taking a limit as $\epsilon_i\to 0$ in \eqref{eq-omegae} we have
\begin{align*}
 \left\langle \prod_{j\in S} \frac{r_j^{s_j}}{s_j!} g, \left(\prod_{j\in S}(T_j^*)^{s_j}
 \right)\left(\chi_S f\right)\right\rangle&= \lim_{i\to \infty}\left\langle \prod_{j\in S} \frac{r_j^{s_j}}{s_j!} g, \left(\prod_{j\in S}(T_j^*)^{s_j}
 \right)\left(\chi_S f\right)\right\rangle_{\Omega_{\epsilon_i}}\\
&=\left\langle T_{j_0}\left( \prod_{j\in S} \frac{r_j^{s_j}}{s_j!} g\right), \left((T_{j_0}^*)^{s_0-1}\cdot\prod_{j\in S\setminus\{j_0\}}(T_j^*)^{s_j}
 \right)\left(\chi_S f\right)\right\rangle\\
&= \left\langle \frac{r_{j_0}^{s_{j_0}-1}}{(s_{j_0}-1)!}\cdot\prod_{j\in S\setminus\{j_0\}} \frac{r_j^{s_j}}{s_j!} g, \left((T_{j_0}^*)^{s_0-1}\cdot\prod_{j\in S\setminus\{j_0\}}(T_j^*)^{s_j}
 \right)\left(\chi_S f\right)\right\rangle,\\
\end{align*}
where in the last line we have used the facts that $T_{j_0} g=0$, and $T_{j_0}r_j=\delta_{jj_0}$.
Repeating the above process $s_0-1$ times more, we conclude that the above expression is equal to 
\[\left\langle \prod_{j\in S\setminus\{j_0\}} \frac{r_j^{s_j}}{s_j!} g, \left(\prod_{j\in S\setminus\{j_0\}}(T_j^*)^{s_j}
 \right)\left(\chi_S f\right)\right\rangle,
\]
and applying the same process to the smallest element of $S\setminus \{j_0\}$ and continuing till we are left with the empty set of indices, we 
conclude that
\[ \left\langle \prod_{j\in S} \frac{r_j^{s_j}}{s_j!} g, \left(\prod_{j\in S}(T_j^*)^{s_j}
 \right)\left(\chi_S f\right)\right\rangle =\langle g, \chi_S f\rangle  \]
We note that the term corresponding to $S=\emptyset$ in \eqref{eq-sum} is simply $\langle g,\chi_\emptyset f\rangle$, and therefore 
we can rewrite \eqref{eq-sum} as:
\begin{align*} 
\langle g, \Phi^s f\rangle&=  \sum_{S\subset\{1,\dots,N\}}\langle g, \chi_S f\rangle\\
&= \left\langle g,\left( \sum_{S\subset\{1,\dots,N\}}\chi_S\right) f\right\rangle\\
&= \langle g,f\rangle,
\end{align*}
since $\{\chi_S\}$ is a partition of unity. This proves the result.
\end{proof}
  Let $K_\Omega$ 
denote the Bergman kernel of a domain $\Omega$. If $P:L^2(\Omega)\to \mathcal{H}(\Omega)\subset L^2(\Omega)$ 
denotes the Bergman projection on $\Omega$,
$K_\Omega$ is characterized by the property that for each $f\in L^2(\Omega)$ we have
\[ P f(w) = \int_\Omega K_\Omega(w,z) f(z) dV(z),\]
where $dV$ is Lebesgue measure. It is well-known that $K_\Omega(w,z)$ is holomorphic in $w$, antiholomorphic in $z$ and satisfies the Hermitian symmetry $K_\Omega(w,z)=\overline{K_\Omega(z,w)}$.
 Lemma~\ref{lem-bell} leads to the following characterization of  Condition R on a  domain with generic corners.
\begin{prop}[cf. \cite{BeBo, Ba1}]\label{prop-condnr}
A domain $\Omega$ with generic corners satisfies Condition R if and only if for each multi-index $\alpha$, there  are constants
$C$ and $m$ depending only on the domain $\Omega$ such that 
\begin{equation}\label{eq-kderestimate} \abs{\frac{\partial^\alpha}{\partial w^\alpha} K_\Omega(w,z)}\leq C \dist(z,b\Omega)^{-m}
\end{equation}
for all $(w,z)\in \Omega\times \Omega$.
\end{prop}
\begin{proof} The method of proof given in \cite{BeBo} may be applied with appropriate minor modification. The crucial point here
is the  existence of the operator $\Phi^s$.
\end{proof}

\section{Some examples}
\label{sec-examples}
We now consider examples of domains $D$ in $\cx^n$
for which the hypotheses  of Theorem~\ref{thm-biholo} hold, i.e., $D$ has generic corners, is pseudoconvex and
satisfies Condition R. 
Note that 
if $D$ satisfies the hypotheses of Theorem~\ref{thm-biholo}, it follows from  Theorem~\ref{thm-biholo} that so does $F(D)$,
where $F:D\to F(D)$ is a biholomorphic map extending smoothly to $\overline{D}$. If $n=1$, the only domains with generic corners 
are the smoothly bounded ones.  For $ n\geq 2$, there do exist domains with generic corners in $\cx^n$. However,  many interesting 
piecewise smooth domains do not have generic corners, e.g., the intersection of two balls in $\cx^2$ (see \cite{barrettvass}.)

For smoothly bounded domains, Condition R is a consequence of global regularity estimates on the $\dbar$-Neumann operator 
(see  \cite{CS, straube} for details.)  Indeed it suffices to know that the $\dbar$-Neumann operator is compact on the space $L^2_{0,1}(D)$
of square integrable  $(0,1)$-forms.  However, as \cite{ehsani} already shows, this strategy is unlikely to succeed with general piecewise smooth domains.  The question of establishing Condition R 
on such domains therefore merits deeper study.  However there are a few cases where Condition R can be established on a domain with generic corners by elementary means.  
\subsection{Products}We first show that the hypotheses of  Theorem~\ref{thm-biholo} propagate to products:
\begin{prop}For $j=1,\dots, k$, let  $D_j\Subset \cx^{n_j}$ be a domain with generic corners which 
satisfies Condition R. Let $n= \sum_{j=1}^k n_j$, and let $D$ be the domain in $\cx^n$ given as $D=D_1\times D_2\times \dots \times D_k$.
Then $D$ has generic corners and satisfies Condition R.
\end{prop}
In \cite{prod, spec}, the following was proved: if $D_1\subset\cx^{n_1},D_2\subset\cx^{n_2}$ are bounded pseudoconvex domains
(no assumption of generic corners on the boundary)
such that each of them satisfies Condition R, then so does their product. Here on the other hand there is no assumption of pseudoconvexity.

Note also that combining this proposition  and Theorem~\ref{thm-biholo} we recapture the famous observation of 
Poincar\'{e}: the ball and bidisc in $\cx^2$ are not biholomorphically equivalent.

\begin{proof}By an induction argument, it is sufficient to prove this for $k=2$. Assume that as in Definition~\ref{defn-main}
we are given representations $D_1=\bigcap_{j=1}^NG_j$ and $D_2=\bigcap_{\ell=1}^M H_k$. We then have 
\begin{align*}D_1\times D_2&= (D_2\times \cx^{n_2})\cap (\cx^{n_1}\times D_2)\\
&= \left(\bigcap_{j=1}^N(G_j\times\cx^{n_2})\right)\cap \left(\bigcap_{\ell=1}^M (\cx^{n_1}\times H_\ell)\right),
\end{align*}
which is a representation of $D_1\times D_2$ as an intersection of smoothly bounded domains. Since $D_1,D_2$ have generic corners,
it is easy to verify that the corners of the product are CR manifolds of the right CR dimension.

Denote by $K_j$ the Bergman kernel of $D_j$. The derivatives of $K_j$ satisfy the estimate \eqref{eq-kderestimate}, since $D_j$ satisfies Condition R.
Thanks to \cite[Theorem~6.1.11]{jp} the Bergman Kernel  $K$  of the product $D_1\times D_2$ can be represented $K_1\tensor K_2$, i.e., for 
$z=(z_1,z_2)\in D_1\times D_2$ and $w=(w_1,w_2)\in D_1\times D_2$, we have
\[ K(w,z)= K_1(w_1,z_1)\cdot K_2(w_2,z_2).\]
Then it follows that the derivatives of $K$ satisfy the estimate \eqref{eq-kderestimate}, and it follows that $D$ satisfies Condition R.
\end{proof}

\subsection{Domains with circular symmetry.}
Recall that a domain $D\subset\cx^n$ is said to be {\em circular} if it is invariant under the natural action of the circle group,
i.e.,  if for each $z\in D$ and each real number $\theta$, 
we have that $e^{i\theta}z\in D$. Clearly, the boundary $bD$ of $D$ has the same circular symmetry. Further, if $D$ has piecewise smooth boundary, 
it is clear that every stratum is invariant under the circle group. Further, we call a domain  {\em complete circular} if  for each $z\in D$,  and for each complex number $\lambda$ in  the
closed unit disc  (i.e. if $\abs{\lambda}\leq 1$), we have $\lambda z\in D$. 

For a piecewise smooth domain $D$, by a {\em face} we mean a connected component of $b\reg{D}$. If $D$ is represented as the intersection
$\cap_{j=1}^N D_j$, where each $D_j$ is smoothly bounded and all intersections of the boundaries are transverse, then it is clear that each face
is a connected component of $bD_j\cap\overline{D}$ for some $j$. The following result,  extending a classical 
argument of Boas and Bell,  gives simple examples of domains with generic corners satisfying 
Condition R:
\begin{prop}[cf. {\cite[Theorem 2$^\prime$]{BeBo}}]\label{prop-examples} Let $D\subset\cx^n$ be a bounded complete circular domain with generic corners
such that for each $\zeta\in bD$,
the radial line from the origin to $\zeta$ meets each face of $bD$ that passes through $\zeta$ transversely. Then $D$ satisfies Condition R.
\end{prop}
We begin by noting a symmetry property of the Bergman kernel of a circular domain:
\begin{lem}\label{prop-circular}
Let $K$ denote the Bergman kernel of a circular domain $D$, where $0\in D$.
Then if $\lambda$ is a complex number and $z,w\in D$ are such that the points
$\lambda w, \overline{\lambda }z$ are in $D$, then we have
\[ K(\lambda w, z)= K(w, \overline{\lambda}z).\]
\end{lem}
\begin{proof}We claim that  there is an orthonormal basis $\{\eta_j\}_{j=1}^\infty$ 
of the Bergman space $\mathcal{H}(D)$ whose elements are 
homogeneous polynomials. Indeed, it is well-known that any holomorphic function on the circular domain $D$ can be expanded in a series of the form
$f(z) = \sum_{k=1}^\infty P_k(z)$,
where each $P_k$ is a homogenous polynomial  and the series converges uniformly on 
compact subsets of $D$ (see e.g.\cite{malgrange}.) Choosing a basis of the space of homogeneous polynomials 
of degree $d$, and taking the union as $d$ ranges over the non-negative integers, we obtain a family of homogeneous polynomials 
whose span is dense in $L^2(\Omega)$.  Further, if $P$ and $Q$ are homogeneous polynomials of degrees
$p$ and $q$ respectively, they are orthogonal in $L^2(D)$ if $p\not=q$. Indeed, if $\theta$ is a real number such that $e^{ i (p-q)\theta}\not=1$, we have
using the change of variables formula and the fact that the unitary transformation $z\mapsto e^{i\theta}z$ has  real Jacobian determinant
 identically equal to 1, 
\begin{align*}\int_D P(z)\overline{Q(z)}dV(z)&= \int_D P(e^{i\theta}w) \overline{Q(e^{i\theta}w)}dV(w)\\
&= e^{i(p-q)\theta}\int_D P(z)\overline{Q(z)}dV(z).
\end{align*}
Consequently, if the Gram-Schmidt process is applied to the spanning family of homogeneous polynomials, it yields the orthonormal sequence $\{\eta_j\}$,
and the Bergman kernel is then represented as $K(w,z)= \sum_{j=1}^\infty \eta_j(w)\overline{\eta_j(z)}$.  Recalling that each $\eta_j$ is homogenous of 
some degree, the result follows.
\end{proof}
We also note the two useful properties  of the Bergman kernel:  the Cauchy-Schwarz inequality
\begin{equation}\label{eq-cauchy-schwarz}
\abs{K(w,z)}\leq \left( K(w,w)K(z,z)\right)^{\frac{1}{2}},
\end{equation}
and the fact that on any bounded domain $D$
\begin{equation}\label{eq-compare}K(z,z)\leq \text{(const)}d(z)^{-n-1},
\end{equation}
 obtained by comparing the Bergman kernel of $D$ with  that of a ball centered at $z$ and radius $d(z)$.
 
\noindent From now on, let $D$ be   complete circular.
Then for a point $w\in D$, we can define the {\em radial boundary distance} $\rho(w)$ in the following way.  Let $w^*$ be the unique point on 
the boundary $bD$ which is collinear with 0 and $w$. We define $\rho (w)=\abs{w^*-w}$. We also denote $d(w)=\dist(w,bD)$, and call this the 
{\em standard boundary distance.}  We will be interested 
in domains in which  there is a constant $C>1$ such that 
\begin{equation}\label{eq-standardradial} \rho(w)\leq C d(w).\end{equation}
Since we always have $d(w)\leq \rho(w)$ we will say that on such domains the radial and standard boundary distances are comparable. 
We first note that this property holds on the domains considered in  Proposition~\ref{prop-examples}:
\begin{lem}Let $D$ be a piecewise smooth complete circular  domain such that for each $\zeta\in bD$, the radial line from $0$ to $\zeta$ meets each 
face of $bD$ which passes through $\zeta$ transversally. Then the standard and the radial distance are comparable on $D$.
\end{lem}
\begin{proof}Let $bD$ be smooth, and fix a tubular neighborhood $U$ of $bD$ in $D$. For a point $z$ in $U$ denote by $\hat{z}$ the 
unique point on $bD$ closest to $z$, and by $z^*$ the point where the radial line from 0 to $z$ meets $bD$. From the transversality of the line $zz^*$
to $bD$, it follows that the angle between $zz^*$ and $z\hat{z}$ is bounded away from $\frac{\pi}{2}$, and the result follows in this case.

Assuming now  that there are at least two faces,
it is sufficient to prove \eqref{eq-standardradial} for $z$ in some neighborhood
$U$ of $bD$ in $D$. Let $U$ be the union of $U_j$, where each $U_j$ is a tubular neighborhood of $bD_j$, where the domain $D$ is
represented as an intersection $\bigcap_{j=1}^N D_j$. Let $\rho_j(z)$ represent the radial distance from $z$ to the boundary $bD_j$,
and $d_j(z)= \dist(z, bD_j)$.
Then if $z\in U_j$, we have $\rho_j(z)\leq C d_j(z)$, where $C$ may be taken independent of $j$. But $\rho(z)\leq \rho_j(z) \leq C d_j(z)$ for each $z$ in 
$U_j$. If a point $z$ in $U$ belongs to more than one $U_j$, it follows that we must have $\rho(z)\leq C \min d_j(z)=C d(z)$ where the minimum is taken over
all $j$ such that the point $z$ belongs to $U_j$. The result is proved.

\end{proof}
The proof is now completed by the following lemma, and an appeal to Proposition~\ref{prop-condnr}:

\begin{lem} Let $D$ be a bounded complete circular  domain in $\cx^n$ . If the standard and radial boundary distances on $D$ are
comparable (i.e., \eqref{eq-standardradial} holds), then $D$ satisfies the estimate \eqref{eq-kderestimate}.
\end{lem}
\begin{proof} Without loss of generality we can assume that the diameter of $D$ is less than or equal to one, since  \eqref{eq-kderestimate}
holds on a domain if and only if it holds on any dilation.
We proceed as in \cite{BeBo}. We fix once for all $z\in D$.  We consider two cases. 
First assume that 
$w\in D$ is such that  $\abs{w}> \frac{1}{2}d(0)$.

We choose a number $0<\delta< \frac{d(0)}{4}$ such that
\[  \frac{4\delta}{d(0)-4\delta}<\frac{1}{2}\rho(z).\]
The number $\delta$ exists since $x\mapsto \frac{4x}{d(0)-4x}$ is increasing on $[0, \frac{d(0)}{4})$.
 Let
 $\lambda=\left(1-\frac{2\delta}{\abs{w}}\right)^{-1}$. For future use we note  that
\begin{align}
\rho(\lambda z)&=\rho(z)-(\lambda-1)\abs{z}\nonumber\\
&=\rho(z)-\frac{2\delta}{\abs{w}-2\delta}\abs{z}\nonumber\\
&\geq \rho(z) - \frac{2\delta}{\frac{\delta(0)}{2}-2\delta}\nonumber\\
&> \frac{1}{2}\rho(z), \label{eq-rholambdaz}
\end{align}
where the last line follows from the choice of $\delta$.
By Lemma~\ref{prop-circular}, 
\begin{align}K(w,z)&=K\left(\lambda^{-1}w, \lambda z\right)\nonumber\\
&= K\left(t, \lambda z\right)\label{eq-est1},
\end{align}
where $t=\left(1-\frac{2\delta}{\abs{w}}\right)w$, and therefore  we have that $\rho(t)\geq 2\delta.$ . Noting that we are considering such $w\in D$ as $\abs{w}>\frac{d(0)}{2}$,   we see that $\abs{\frac{\partial^\alpha}{\partial w^\alpha}\left(\frac{1}{\abs{w}}\right)}$, and therefore 
$\abs{\frac{\partial^\alpha}{\partial w^\alpha}\left(\frac{w_j}{\abs{w}}\right)}$ are bounded (the latter for each $j$.)  It follows that for any 
multi-index $\alpha$ with $\abs{\alpha}\geq 2$, we have that $\abs{\frac{\partial^{\alpha}}{\partial w^\alpha}t_j}\leq C \delta$,
 and that  $\abs{\frac{\partial^{\alpha}}{\partial w^\alpha}\lambda}\leq C \delta$, where here and in the sequel the constant $C$ depends on $\alpha$ but is independent of $z$ (and 
therefore $\delta$)  and $w$ (with $\abs{w}> \frac{d(0)}{2}$), but $C$ may have different values at different occurences. Using  the alternative representation of 
$K(w,z)$ in \eqref{eq-est1},  and the repeated 
use of the chain and the product rule (i.e. the Faa di Bruno formula), one may compute an expression for $\frac{\partial^\alpha}{\partial w^\alpha}K(w,z)$, in terms of the $t$ and $z$-derivatives of $K(t,z)$ and the $w$ derivatives of $t$ and $\lambda$. It follows  that 
\begin{equation}\label{eq-derk} \abs{\frac{\partial^\alpha}{\partial w^\alpha}K(w,z)}\leq C \delta\cdot \sum_{\beta+\gamma\leq \alpha} \abs{\frac{\partial^\beta}{\partial t^\beta}\frac{\partial^\gamma}{\partial \overline{z}^\gamma} K\left(t, \lambda z\right)},\end{equation}
since higher powers of $\delta$ may be absorbed into $\delta$ itself (since $\delta<\frac{1}{4}$.) Now, thanks to the comparability of the standard and radial distances to the boundary, we see that there is a polydisc of polyradius $C(\delta,\delta,\dots,\delta)$
with center at $t=\left(1-\frac{2\delta}{\abs{w}}\right)w$
 and located within $\{\zeta\in D\colon \rho(\zeta)>\delta\}$. Recalling that $K$ is holomorphic in the first and antiholomorphic in the second argument, and applying the Cauchy estimates in both arguments to this polydisc we conclude:
 
 \begin{align*}
 \abs{\frac{\partial^\beta}{\partial t^\beta}\frac{\partial^\gamma}{\partial \overline{z}^\gamma }K\left(t, \lambda z\right)}&\leq \frac{C}{\delta^{\abs{\beta}+\abs{\gamma}}}\sup_{\rho(t)>\delta}\sqrt{K(t,t)} \sqrt{K\left(\lambda z,\lambda z\right) }&\text{ using \eqref{eq-cauchy-schwarz}}\\
&\leq \frac{C}{\delta^{\abs{\beta}+\abs{\gamma}}}\delta^{-\frac{n+1}{2}}\rho(\lambda z)^{-\frac{n+1}{2}}&\text{ using \eqref{eq-compare}}\\
&\leq \frac{C}{\delta^{\abs{\beta}+\abs{\gamma}}}\delta^{-\frac{n+1}{2}} \rho(z)^{-\frac{n+1}{2}}&\text{ using \eqref{eq-rholambdaz}}\\
&\leq \frac{C}{d(z)^{\abs{\beta}+\abs{\gamma}+n+1}}.
 \end{align*}
Combining this with \eqref{eq-est1}, we conclude that
\[ \abs{\frac{\partial^\alpha}{\partial w^\alpha}K(w,z)}\leq \frac{C}{d(z)^{\abs{\alpha}+n}},\]
for $w$ such that $\abs{w}> \frac{d(0)}{2}$.

We now consider a $w\in D$ such that $\abs{w}\leq \frac{d(0)}{2}$.  Then there is an $\eta$ independent of $w$ such that $\rho(w)>\eta$. By the 
comparability of $\rho$ and $d$, we conclude that there is an $\epsilon>0$ such that the polydisc centered at $w$ and of radius $\epsilon$ is 
contained in the set $\{\zeta\in D\colon d(\zeta)>\epsilon\}$. (Note that $\epsilon$ depends only on $d(0)$ and the constant $C$ in 
\eqref{eq-standardradial}.) Applying the Cauchy estimates to this polydisc we see that
\begin{align*}\abs{\frac{\partial^\alpha}{\partial w^\alpha}K(w,z)}&\leq \frac{C}{\epsilon^{\abs{\alpha}}}\sup_{d(w)>\epsilon}\abs{K(w,z)}\\
&\leq  C\sqrt{K(z,z)}\sup_{d(w)>\epsilon}\sqrt{\abs{K(w,w)}}\\
&\leq Cd(z)^{-\frac{(n+1)}{2}}\epsilon^{-\frac{(n+1)}{2}},
\end{align*}
therefore the estimate \eqref{eq-kderestimate} is established and the result is proved.
\end{proof}
\section{Hopf Lemma on domains with generic corners }

\noindent Let $D \subset \mathbb C^n$ be a smoothly bounded domain and $\phi : D \rightarrow [-\infty, 0)$ a plurisubharmonic 
exhaustion function. The Hopf lemma asserts that $\vert \phi(z) \vert$ decays to zero near the boundary $bD$ at least at the rate of ${\rm dist}(z, bD)$, i.e.,
\begin{equation}
\vert \phi(z) \vert \gtrsim {\rm dist}(z, bD)
\end{equation}
for all $z \in D$. For a given proper holomorphic mapping $f : D \rightarrow G$, this estimate plays a useful role in controlling the ratio ${\rm 
dist}(f(z), bG)/{\rm dist}(z, bD)^{\eta}$ for some $\eta > 0$. Thus we are interested in obtaining (4.1) on non-smooth domains as well. For piecewise smooth domains, this was done in 
\cite{Be,Ra} by showing that each point sufficiently close to the boundary lies in a cone of uniform aperture with vertex on the boundary. In other words, a planar sector of 
uniform aperture containing  a given point near the boundary was shown to exist. On a product domain $G$, it is evident that a sector whose aperture angle is $\pi/2$, i.e., a quadrant, 
can be fitted at each boundary point. Therefore the techniques of \cite{Be} show that a negative plurisubharmonic exhaustion $\phi$ on a product domain satisfies
\[
\vert \phi(z) \vert \gtrsim {\rm dist}(z, b G)^{2} 
\]
for all $z \in G$. A different approach was used in \cite{prop} for product domains wherein a disc that satisfies certain uniform geometric properties 
was used instead of a sector. Similar ideas can be applied to domains with generic corners  as well which yield a better growth estimate.

\begin{prop}
Let $\Omega \subset \mathbb C^n$ be a domain with generic corners . Let $\phi : \Omega \rightarrow [-\infty, 0)$ be a plurisubharmonic exhaustion. Then 
\[
\vert \phi(z) \vert \gtrsim {\rm dist}(z, bD)
\]
for all $z \in \Omega$.
\end{prop}

\noindent We first recall some geometric conditions on an analytic disc from \cite{prop} that are sufficient to prove (4.1). Let $D \subset \mathbb C^n$ be a bounded domain and take a 
tubular neighborhood $U$ of $bD$. The domain $U \cap D$ whose boundary consists of two disjoint components, namely  $bD$ and $B = bU \cap D$ will be relevant to us. Suppose that there 
is a constant $\theta = \theta(D) \in (0, 2 \pi)$ and points $\kappa(z) \in B, \zeta(z) \in bD$ (both possibly non-unique) for every $z \in U \cap D$ such 
that the following hold:
\begin{itemize}
\item[(i)] The points $\zeta(z), z, \kappa(z)$ are collinear and $z$ lies between $\zeta(z)$ and $\kappa(z)$. 
\item[(ii)] $\zeta(z)$ is the nearest point to $z$ on $bD$ which means that $\vert \zeta(z) - z \vert = {\rm dist}(z, bD)$.
\item[(iii)] The affine analytic disc $\alpha_z : \Delta(0, 1) \rightarrow \mathbb C^n$ given by
\[
\alpha_z(\lambda) = \kappa(z) + \lambda(\zeta(z) - \kappa(z))
\]
lies in $D$.
\item[(iv)] There exists a neighborhood of $\partial \Omega$ in $\mathbb C^n$, say $V$ which is compactly contained in $U$ such that the portion of the 
boundary of $\alpha_z(\Delta(0, 1))$, i.e., $\alpha_z(b\Delta(0, 1))$ that lies in $D \setminus V$ subtends an angle of at least $\theta = \theta(D) > 0$ 
at the centre $\kappa(z)$. Note that $\alpha_z(0) = \kappa(z)$ and $\alpha_z(1) = \zeta(z)$.
\end{itemize}
In short, these properties allow the existence of an analytic disc passing through a given point $p$ near $bD$ and also containing $p^{\ast}$, a nearest point to $p$ on $bD$, whose 
centre is at a uniform distance away from $bD$ and such that a uniform piece of its boundary is also uniformly away from $bD$. We say that it is possible to {\em roll an analytic disc} 
in $D$ if these properties hold. Theorem 4.4 in \cite{prop} shows that the Hopf lemma holds on a domain if it is possible to roll an analytic disc in it.

\begin{proof}
It suffices to show that it is possible to roll an analytic disc in a domain $\Omega$ with generic corners  as in Definition 1.1. Fix a point $p \in b \Omega$ and let $S \subset 
\{1, 2, \ldots, N\}$ be such that 
\[
p \in B_S = \bigcap_{j \in S} b \Omega_j.
\]
Without loss of generality we may assume that $S = \{1, 2, \ldots, k\}$ where $k \le N$. Then
\[
r_1(p) = r_2(p) = \ldots = r_k(p) = 0
\]
and (2.1) holds. Thanks to this transversality condition, we may choose coordinates in a neighborhood $U$ around $p = 0$ so that the defining functions become
\[
r_j(z) = 2 \Re z_j + \phi_j(z)
\]
where $\phi_j \in \mathcal C^{\infty}(U)$ and $d \phi_j(0) = 0$ for all $1 \le j \le k$. The smoothness of each $r_j$ implies that for a given point $z \in U$ there is a unique point 
$z^{\ast}_j$ on $\{r_j = 0\} = b \Omega_j \cap U$ such that
\[
\tau_j = {\rm dist}(z, b \Omega_j \cap U) = \vert z - z^{\ast}_j \vert
\]
for all $1 \le j \le k$. The analytic disc
\[
\zeta \mapsto z + \zeta \tau_j \left( {\overline \partial} r_j(z^{\ast}_j) \right)
\]
for $\vert \zeta \vert < 1$ is centered at $z$ and is contained in $\{r_j < 0\} = \Omega_j \cap U$. Thus through a given point $z \in \Omega \cap U$ there are $k$ analytic discs which 
approximately point in the direction of the coordinate axes $z_1, z_2, \ldots, z_k$. This observation will allow us to choose the right direction for the disc $\alpha_z(\lambda)$ as in 
(iii) above. Let $C > 0$ be such that
\begin{equation}
C^{-1} \vert r_j(z) \vert \le {\rm dist}(z, b \Omega_j \cap U) \le C \vert r_j(z) \vert
\end{equation}
for all $z \in \Omega \cap U$ and $1 \le j \le k$. Furthermore, since $d \phi_j(0) = 0$ we may also assume that $\vert d \phi_j(z) \vert \le 1/2C$ for all $z \in \Omega \cap U$. For 
$\epsilon > 0$ let
\[
\Omega_{\epsilon} = \{ z \in U : r_j(z) < -\epsilon, \; 1 \le j \le k \}.
\]
Pick $z \in U \cap \Omega$ and note that the nearest point to it (which is possibly non-unique) on $b \Omega \cap U$ lies on one or possibly more of the boundaries $b \Omega_j \cap U = 
\{r_j = 0\}$. For the sake of definiteness, assume that it lies on $b \Omega_1 \cap U = \{r_1(z) =0\}$ and denote it by $\zeta(z)$. Extend the real inner normal $l$ to the smooth real 
hypersurface $\{ r_1(z) = 0\}$ at $\zeta(z)$ till it intersects $b \Omega_{\epsilon} \cap U$. Denote this point of intersection by $\kappa(z)$. Note that $\abs{\kappa(z) - \zeta(z)} = 
{\rm dist}(\kappa(z), \{r_1(z) = 0\} \cap U) \ge \epsilon/C$ by (4.2). The affine analytic disc
\[
\alpha_z(\lambda) = \kappa(z) + \lambda(\zeta(z) - \kappa(z))
\]
defined for $\vert \zeta \vert < 1$ is evidently contained in $\{r_1 < 0\} \cap U$. For $1 < j \le k$ observe that
\[
\abs{r_j(\alpha_z(\lambda)) - r_j(\kappa(z))} = \vert \lambda(\zeta(z) - \kappa(z)) \cdot dr_j(\tilde z) \vert
\]
for some $\tilde z = \tilde z(\lambda) \in \Omega \cap U$. Again (4.2) shows that $\abs{\kappa(z) - \zeta(z)} \le C \epsilon$ and by construction we have $\vert d r_j(\tilde z) 
\vert \le 1/2C$. Combining these estimates shows that
\[
r_j(\alpha_z(\lambda)) \le r_j(\kappa(z)) + \epsilon/2 \le - \epsilon + \epsilon/2 = -\epsilon/2.
\]
Thus the analytic disc $\alpha_z(\lambda)$ is contained in $\{r_1 < 0\} \cap U$ and stays at a uniform distance from the other hypersurfaces $\{ r_j = 0\} \cap U$ where $1 < j \le k$. 
Let $V = \{z \in U : \vert r_j(z) \vert < \epsilon/2, \; 1\le j \le k \}$ -- this is a neighborhood of $b \Omega \cap U$ of uniform width $\epsilon/2$. The smoothness of $r_1$ 
shows that there is a uniform portion of $b \alpha_z(\lambda)$ that lies in $(\Omega \cap U) \setminus \{r_1 > -\epsilon/2\}$. The arguments given above show that the closure of 
$\alpha_z(\lambda)$ lies in $(\Omega \cap U) \setminus \{r_j > - \epsilon/2, \; 1 < j \le k \}$ and hence a uniform portion of $b \alpha_z(\lambda)$ lies in $(\Omega \cap U) 
\setminus V$. These estimates are uniform for all $z \in \Omega \cap U$ and hence for all $z$ near $b \Omega$ by compactness. Hence it is 
possible to roll an analytic disc in $\Omega$.
\end{proof}

\section{Proper maps of domains with generic corners}
\subsection{Distortion estimate on domains with generic corners }
We now generalize some well-known properties of proper maps of smoothly bounded pseudoconvex domains
to domains with generic corners.  In these results, $D$ and $G$ are  pseudoconvex domains with generic corners, and $f:D\to G$ is a proper holomorphic 
mapping.  Let  $Z= \{f(z) : \det f'(z) = 0 \} \subset G$ be the set of critical values of $f$.
Then $Z$ is a codimension one subvariety in $G$, and on $G \setminus Z$, we can define 
locally well-defined holomorphic branches  $F_1, F_2, \ldots, F_m$ of $f^{-1}$.
The following consequence of the  Hopf lemma is well-known in the case of smoothly bounded domains.
\begin{prop}\label{prop-distest}  There exists a $\delta \in (0, 1)$ such that
\[
{\rm dist}(z, b D)^{1/\delta} \lesssim {\rm dist}(f(z), b G) \lesssim {\rm dist}(z, b D)^{\delta}
\]
for all $z \in D$. 
\end{prop}

\begin{proof} We begin by noting that if $\Omega$ is a pseudoconvex domain with generic corners, then there is a  negative strictly \psh exhaustion $\varrho$
of $\Omega$  which decays to  to zero at the boundary no faster than a power of the distance to the boundary, i.e., for some $0 < \eta < 1$
and all $z\in \Omega$ we have 
\[ \abs{\varrho(z)}\lesssim \dist(z,b\Omega)^{\eta}.\]
This follows directly (even for Lipschitz $\Omega$) from \cite{harrington}. We can also deduce it from the fact that if  as in
 Definition~\ref{defn-main},the domain $\Omega$ is represented as an intersection $\cap_{j=1}^N \Omega_j$ of smoothly bounded pseudoconvex
domains, then by famous results of Diederich and Fornaess \cite{difo}, each $\Omega_j$ admits a bounded \psh exhaustion $\varrho_j$ 
satisfying $\abs{\varrho_j(z)}\lesssim \dist(z, b\Omega_j)^{\eta_j}$ for some $\eta_j\in (0,1)$ and for each $z\in \Omega_j$. We can simply 
take $\varrho = \max_{1\leq  j\leq N}\varrho_j$.

\medskip

Therefore let $\varrho_D$ and $\varrho_G$ be  bounded plurisubharmonic exhaustions on $D$ and $G$ such that  for some $\eta, \tau\in (0,1)$ 
and 
\[
\abs{ \varrho_D(z) } \lesssim {\rm dist}(z, b D)^{\eta}
\]
for all $z \in D$ , and 
\[
\abs{ \varrho_G(w)} \lesssim {\rm dist}(z, b G)^{\tau}
\]
for all $w \in G$. Then $\varrho_G \circ f$ is a negative plurisubharmonic exhaustion on $D$ and satisfies
\[
- \varrho_G \circ f(z) = \vert \varrho_G \circ f(z) \vert \gtrsim{\rm dist}(z, b D)
\]
for all $z \in D$ by the Hopf lemma. Thus we get
\[
{\rm dist}(z, b D) \lesssim - \varrho_G \circ f(z) \lesssim {\rm dist}(f(z), b G)^{\tau}
\]
which is the left side inequality in the proposition.

\medskip

Recall that $F_1,\dots F_m$ denote the branches of the inverse mapping $f^{-1}$, which are locally well-defined on $G\setminus Z$, where $Z$ is the set of critical values of the
mapping $F$. Then
\[
\psi = \max\{ \varrho_D \circ F_j : 1 \le j \le m\}
\]
is a bounded continuous plurisubharmonic function on $G \setminus Z$ which extends to a plurisubharmonic exhaustion on $G$. Therefore, for each $1 \le j \le m$ and $w \in G$, we have
\[
-\varrho_D \circ F_j(w) \ge -\psi(w) = \vert \psi(w) \vert \gtrsim {\rm dist}(w, b G)
\]
where the last inequality follows from the Hopf lemma. Rewriting this as
\[
\vert \varrho_D(z) \vert = -\varrho_D(z) \gtrsim {\rm dist}(f(z), b G)
\]
and combining with the rate of decay of $\varrho_D$ near $b D$ we get 
\[
{\rm dist}(f(z), b G) \lesssim {\rm dist}(z, b D)^{\eta}
\]
for all $z \in D$ which completes the proof.
\end{proof}

\subsection{Smoothness of the Jacobian up to the boundary}
We now note that the following lemma, well-known for smoothly bounded domains,  continues to hold for domains with generic corners.
For a domain $\Omega$ in $\cx^n$, we denote by $\mathcal{H}^\infty(\Omega)$ the space $\mathcal{O}(\Omega)\cap \mathcal{C}^\infty(\overline{\Omega})$ 
of holomorphic functions on $\Omega$ which are smooth up to the boundary of $\Omega$.

\begin{lem}\label{lem-jacobian}
Suppose that $D$ satisfies Condition R, and  let $u=\det(f\rq{})$ be the Jacobian determinant  of the mapping $f$.
If  $h\in \mathcal{H}^\infty(G)$, we have
\[ u \cdot (h\circ f)\in \mathcal{H}^\infty(D).\]
\end{lem} 
\begin{proof} We adapt the classical proof from \cite{Be1}. 
Let $\ell$ be a given positive integer. We need to show that $u \cdot (h\circ f)\in \mathcal{C}^\ell(\overline{D})$.
Denote by $P$ and $Q$ the Bergman projections on the domains $D$ and $G$ respectively. 
Now, thanks to the classical
transformation formula  for the Bergman projection, we have for each $g\in L^2(G)$ that
\[ P(u\cdot (g\circ f))= u\cdot (Q(g)\circ f).\]
For an $N$-tuple $s=(s_1,\dots, s_N)$ of positive integers,
let $\Phi^s$ be the operator on $G$   as constructed in Lemma~\ref{lem-bell}, and set $g_s= \Phi^s h$. 
Then $Qg_s=h$, and we have
\[ u \cdot (h\circ f)=  P(u\cdot (g_s\circ f))\]
Since $D$ satisfies Condition R, it follows that
there is an integer $k$ such that $P$ maps $\mathcal{C}^k(\overline{D})$ into $\mathcal{C}^\ell (\overline{D})$.
Therefore, to prove the result, it suffices to show that there is a tuple $s$ such that $u\cdot (g_s\circ f)\in \mathcal{C}^k(\overline{D})$.
It will be sufficient to show that derivatives of order $k+1$ of the function  $u\cdot (g_s\circ f)$ on $D$ are all bounded.

\medskip

Denote the map $f$ in components as $f=(f_1,\dots,f_n)$, where
each $f_j$ is complex valued on $D$.  Note that each $f_j$ is bounded. Consequently we have the Cauchy estimates
\[ \abs{D^\alpha f_j(z)} \lesssim \dist(z,b D)^{-\abs{\alpha}},
\]
and 
\[ \abs{D^\alpha u(z)} \lesssim \dist(z,b D)^{-\abs{\alpha} - n}.
\]
We will take the tuple $s$ be to of the form $s=(\sigma,\dots,\sigma)$, i.e., all $N$  entries are equal to the same
positive integer $\sigma$. If $N\sigma>\abs{\alpha}$, and $w\in G$, by Lemma~\ref{lem-bell} we have an estimate
\[ \abs{D^\alpha(g_s(w))}\lesssim  \dist(w, bG)^{N\sigma - \abs{\alpha}}.\]
Note that we have
\[ D^\alpha(g_s\circ f)(z) = \sum D^\beta g_s(f(z)) D^{\delta_1} f_{i_1}D^{\delta_2}f_{i_2} \dots D^{\delta_p} f_{i_p},\]
where the sum extends over all tuples $1\leq i_1,\dots, i_p\leq n$, and multi-indices $\beta$, $\delta_1,\dots, \delta_p$
with $\abs{\beta}\leq \abs{\alpha}$ and $\abs{\delta_1}+\dots+\abs{\delta_p}=\abs{\alpha}$. Therefore we have
the estimate
\begin{align*} \abs{D^\alpha(g_s\circ f)(z) } &\lesssim \dist(f(z),bG)^{N\sigma-\abs{\alpha}}\cdot \dist(z, bD)^{-(\abs{\alpha}+1)}\\
& \lesssim  \dist(z, bD)^{\delta(N\sigma-\abs{\alpha})} \cdot \dist(z, bD)^{-(\abs{\alpha}+1)}\\
& \lesssim\dist(z, bD)^{\delta N \sigma-(1+\delta)\abs{\alpha}-1},
\end{align*}
where in the second line, the right half of the  distortion estimate from Proposition~\ref{prop-distest} has been used. It follows that by taking $\sigma$ 
to be sufficiently large, we can make the function $ D^\alpha(g_s\circ f)$ vanish to arbitrarily high order on the boundary  $bD$.
Using the Cauchy estimates on the derivatives of $u$, and the Leibniz rule for the derivative of a product it now follows that
by taking $\sigma$ sufficiently large, we can ensure that for any multi-index $\alpha$, the derivative $D^\alpha(u\cdot (g_s\circ f))$ 
is bounded on $D$. Consequently there is an $N$-tuple $s$ such that $u\cdot (g_s\circ f)\in \mathcal{C}^k(\overline{D})$, and our result is proved.
\end{proof}

\subsection{Smoothness to the boundary of symmetric functions of the inverse branches}
\begin{prop} If $G$ is as in Theorem~\ref{thm-proper},  i.e., $G$ is a product of smoothly bounded
domains, then for $h\in \mathcal{H}^\infty(D)$, an elementary symmetric function of $h\circ F_1, h\circ F_2,\dots, h\circ F_m$
(defined on $G\setminus Z$) extends to a function in $\mathcal{H}^\infty(G)$.
\end{prop}
\begin{proof}
For smoothly bounded domains $D$ and $G$, this is a classical result of Bell (see \cite{Be1}.) It was shown in \cite[Proposition~5.3]{prop} that the same arguments, with minor modifications work when each of $D$ and $G$ is a product of  smoothly
 bounded domains.  We note here that the proof given in \cite{prop} actually works in the more general situation when $D$ is merely a domain with generic corners and is not necessarily a product.
\end{proof}

\section{Proof of Theorem~\ref{thm-biholo}}

\subsection*{1$\Rightarrow$2} Since $f$ maps $bD$ diffeomorphically to $bG$, it follows that $G$ must have piecewise smooth boundary.
Since the map $f$ is CR on each of the manifolds constituting $bD$, it follows that  $G$ is a domain with generic corners.

Let $g=f^{-1}$, and let $K_G(z,w)$ and $K_D(Z,W)$ denote the Bergman kernels on the domains $G$ and $D$ respectively. Since by hypothesis,
$D$ satisfies Condition R, it follows from Proposition~\ref{prop-condnr} that for each multi-index $\alpha$, there is an $m_\alpha$ such that 
 we have an estimate
\[ \abs{\left(\frac{\partial}{\partial W}\right)^\alpha K_D(W,Z)}\lesssim \dist(Z, bD)^{-m_\alpha},\]
valid for all $(W,Z)\in D\times D$. The kernels $K_D$ and $K_G$  are related by the classical
formula (\cite{Be1} or \cite[Proposition~6.1.7]{jp})
\[ K_G(w,z)= \det g'(w) K_D\left(g(w), g(z)\right) \overline{\det g'(z)}.\]
Therefore
\[\left(\frac{\partial}{\partial w}\right)^\alpha  K_G(w,z)= \left( \left(\frac{\partial}{\partial w}\right)^\alpha \det g'(w) K_D\left(g(w), g(z)\right)\right)\cdot \overline{\det g'(z)}.\]
Since $g$ is smooth up to the boundary,  
\begin{equation}\label{eq-gprimeest} \abs{\det g\rq{}(z)}\lesssim 1.\end{equation}

\noindent Again, since $g$ is smooth up to the boundary, by repeated application of the chain rule and the product rule we obtain 
\begin{align*}
\abs{ \left(\frac{\partial}{\partial w}\right)^\alpha \left(\det g'(w) K_D\left(g(w), g(z)\right)\right)} &\lesssim \sum_{\abs{\beta}\leq \abs{\alpha}} 
\abs{\left(\frac{\partial}{\partial W}\right)^\beta K_D(g(w),g(z))} \\
& \lesssim \sum_{\abs{\beta}\leq \abs{\alpha}} \dist(g(z), bD)^{-m_\beta}\\
&\lesssim \dist(g(z), bD)^{-M} \text{\quad ($M$ being the largest of the $m_\beta$\rq{}s)}\\
&\lesssim \dist(z, bG)^{-\frac{M}{\delta}},
\end{align*}
where in the last line we have used Proposition~\ref{prop-distest}.   Combining this with \eqref{eq-gprimeest}, and invoking
Proposition~\ref{prop-condnr} our result follows.

\subsection*{2$\Rightarrow$1} 

Taking $h\equiv 1$ in Lemma~\ref{lem-jacobian}, we see that $u\in \mathcal{C}^\infty(\overline{D})$.
Applying the lemma again to the mapping $f^{-1}:G\to D$, we obtain that $\det ((f^{-1})\rq{})\in \mathcal{C}^\infty(\overline{G})$. But this implies that
$u^{-1}\in \mathcal{C}^\infty(\overline{D})$. It therefore follows that for each holomorphic $h$ on $G$ such that 
$h\in \mathcal{C}^\infty(\overline{G})$, we have that $h\circ f\in \mathcal{C}^\infty(\overline{D})$. Taking $h$ to be the coordinate functions $z\mapsto z_j$ from $D$ to $\cx$,
we see that each  component of $f$  extends smoothly to the boundary, and the result is proved.

\section{Proof of Theorem~\ref{thm-proper}} 
\noindent The proof of  Theorem~\ref{thm-proper} is for most part identical to the first part of the argument for the proof of  \cite[Theorem~1.1]{prop},
where it is further assumed that $D$ is also  a product domain. We review the main steps of the proof below, noting in each step that the 
hypothesis of product structure is not really used in the proof of continuous extension to the boundary. (It does become relevant in the latter
part of the proof of \cite[Theorem~1.1]{prop}, i.e.,  Lemma~5.7 onward.) What is important is that $D$ is piecewise smooth, pseudoconvex, 
satisfies Condition R, and there is a Bell operator on $D$.

\medskip

As in Lemma~\ref{lem-jacobian},  let $u= \det(f')$ be  the Jacobian determinant of  the map $f:D\to G$. We claim that $u$ vanishes to at most finite order 
at each point of $\partial D$. For smoothly bounded domains, the proof can be found in \cite{Be2,BL}.  It was shown in \cite[Lemma~5.5]{prop} 
that essentially the same argument continues to work for the piecewise smooth domains considered here.

\medskip

From this, as in  \cite[Lemma~5.6]{prop}, it follows that $f$ extends to a continuous map from $\overline{D}$ to $\overline{G}$. The key ingredient 
here is the weak division result \cite[Lemma~10]{DF}  which states the following: On a smoothly bounded domain $\Omega \subset \mathbb C^n$, let $u \in \mathcal 
H^{\infty}(\Omega)$ be a function that does not vanish to infinite order at any point on $b \Omega$. If $h$ is a bounded holomorphic function on $\Omega$ such that $u \cdot h^N \in 
\mathcal H^{\infty}(\Omega)$ for all $N \ge 1$, then $h$ is continuous on $\overline \Omega$. To prove that $h$ is continuous at $p \in b \Omega$, the only geometric requirement is the 
existence of a 
complex line through $p$ that enters $\Omega$ near $p$ and which is transverse to $b \Omega$ near $p$. This condition is clearly satisfied at all points of $b \reg{D}$ while at the 
generic corners such a complex line may be chosen to be transverse to the tangent cone to $b \Omega$ at such points. Thus the proof of \cite[Lemma~5.6]{prop} carries over to the case of 
domains with generic corners. To show that $f$ is smooth at all points of $b \reg{D}$, the finite order vanishing of $u$ at the boundary can be combined with the strong form of the 
division theorem (which is a local statement) as in \cite{BC} or \cite{DF}. This completes the proof of Theorem 1.3.


\end{document}